\documentclass[a4paper,12pt,reqno]{amsart}

\usepackage{epsfig}
\usepackage{amsthm,amsfonts}
\usepackage{amssymb,graphicx,color}
\usepackage[all]{xy}
\usepackage{verbatim}
\usepackage{hyperref}
\usepackage{enumitem}

\usepackage[a4paper,top=2cm,bottom=2cm,left=2.5cm,right=2.5cm]{geometry}

\newtheorem{theorem}{Theorem}[section]
\newtheorem*{theorem*}{Theorem}
\newtheorem{lemma}[theorem]{Lemma}
\newtheorem{corollary}[theorem]{Corollary}
\newtheorem{proposition}[theorem]{Proposition}
\newtheorem{definition}[theorem]{Definition}

\newtheorem{example}[theorem]{Example}
\newtheorem{remark}[theorem]{Remark}

\newcommand{\R}{\mathbb{R}}

\newcommand{\C}{\mathbb{C}}

\begin{document}

\title[H\"older equivalence of complex analytic curve singularities]
{H\"older equivalence of complex analytic curve singularities}

\author{Alexandre Fernandes}
\author{J. Edson Sampaio}
\address{Alexandre ~Fernandes and J. Edson ~Sampaio -
Departamento de Matem\'atica, Universidade Federal do Cear\'a, Av.
Humberto Monte, s/n Campus do Pici - Bloco 914, 60455-760
Fortaleza-CE, Brazil}
 \email{alex@mat.ufc.br}
 \email{edsonsampaio@mat.ufc.br}

\author{Joserlan P. Silva}
\address{Joserlan P. ~Silva -
Instituto de Ci\^encias Exatas e da Natureza, Universidade da Integra\c c\~ao Internacional da Lusofonia Afro-Brasileira, CE060, Km51, Campus dos Palmares, 62785-000
Acarape-CE, Brazil}
 \email{joserlanperote@unilab.edu.br}

\thanks{The first named author were partially supported by CNPq-Brazil grant 302764/2014-7 \\
The third named author were partially supported by CAPES and CNPq-Brazil}

\begin{abstract} We prove that if two germs of irreducible complex analytic curves at $0\in\C^2$ have different sequence of characteristic exponents, then there exists  $0<\alpha<1$ such that those germs are not $\alpha$-H\"older homeomorphic.
For germs of complex analytic plane curves with several irreducible components we prove that if any two of them are $\alpha$-H\"older homeomorphic, for all $0<\alpha<1$, then there is a correspondence between their branches preserving sequence of characteristic exponents and intersection multiplicity of pair of branches. In particular, we recovery the sequence of characteristic exponents of the branches and intersection multiplicity of pair of branches are Lipschitz invariant of germs of complex analytic plane curves.

\end{abstract}

\maketitle
\section{Introduction}

The recognition problem of embedded topological equivalence of germs of complex analytic plane curves at $0\in\C^2$ has a complete solution due to K. Brauner, W. Burau, Kh\"aler and O. Zariski (See \cite{B-K}). For instance, for irreducible germs (branches), it is shown that any two of them are topological equivalent if, and only if, they have the same sequence of characteristic exponents.  For germs of complex analytic plane curves with several irreducible components (several branches), it is shown that any two of them  are topological equivalent if, and only if, there is a correspondence between their branches preserving sequence of characteristic exponents of branches and intersection multiplicity of pair of branches. In \cite{P-T}, F. Pham and B. Teissier proved that if two germs of complex analytic plane curves at $0\in\C^2$, let us say $X$ and $Y$, are topological equivalent as germs embedded in $(\C^2,0)$ (i.e. there exists a bijection between their branches preserving sequence of characteristic exponents and intersection multiplicity of pair of branches), then there exists a germ of meromorphic bi-Lipschitz homeomorphism $$\phi\colon(\C^2,0)\rightarrow(\C^2,0)$$ such that $\phi(X)=Y$. Actually, Pham and Teissier proved the respective converse result exactly as it is stated below. Other versions of this result can be seen in \cite{F} and \cite{N-P}.

\begin{theorem}[Pham-Teissier] If there exists a germ of meromorphic bi-Lipschitz homeomorphism $\phi\colon X\rightarrow Y$ (not necessarily from $\C^2$ to $\C^2$), then there exists a correpondence between their branches preserving sequence of characteristic exponents and intersection multiplicity of pair of branches.
\end{theorem}

In the next, we are going to define the notion of $\alpha$-H\"older equivalence
of germ of subsets in Euclidean spaces, where $\alpha$ is a positive real number. Let us remind that a mapping $f\colon U\subset\R^n\rightarrow\R^m$ is called \emph{$\alpha$-H\"older} if there exists a positive real number $C$ such that
$$\| f(p)-f(q) \| \leq C \| p-q \|^{\alpha} \ \forall \ p,q\in U.$$

\begin{definition}{\rm  Let $(X,x_0)$ and $(Y,y_0)$ be germs of Euclidean subsets. We say that $(X,x_0)$ is \emph{$\alpha$-H\"older homeomorphic to} $(Y,y_0)$ if there exists a germ of homeomorphism $f\colon (X,x_0)\rightarrow (Y,y_0)$ such that $f$ and its inverse $f^{-1}$ are $\alpha$-H\"older mappings. In this case, $f$ is called a \emph{bi-$\alpha$-H\"older homeomorphism} from $(X,x_0)$ onto $(Y,y_0)$.}
\end{definition}

\begin{remark} {\rm Let us remark that bi-$1$-H\"older homeomorphisms from $(X,x_0)$ onto $(Y,y_0)$ are nothing else than bi-Lipschitz homeomorphisms. Moreover, if  $(X,x_0)$ is $\alpha_0$-H\"older homeomorphic to $(Y,y_0)$ for some $0<\alpha_0\leq 1$, then $(X,x_0)$ is $\alpha$-H\"older homeomorphic to $(Y,y_0)$ for any $0<\alpha\leq\alpha_0$}.
\end{remark}

One of the goals of this paper is to prove that for any pair $X$ and $Y$ of germs of irreducible complex analytic  curves at $0\in\C^2$ with different sequence of characteristic exponents, there exists $0<\alpha<1$ such $X$ and $Y$ are not $\alpha$-H\"older homeomorphic. For germs of complex analytic plane curves at $0\in\C^2$, $X$ and $Y$, with two branches, we prove that if the contact of their branches are different, then there exists  $0<\alpha<1$ such $X$ and $Y$ are not $\alpha$-H\"older homeomorphic. Let us remark that these results generalize Theorem of Pham-Teissier and its versions in \cite{F} and \cite{N-P}, see Corollary \ref{cor-2} and \ref{cor-3}.

\bigskip


\section{Preliminaries}\label{section:preliminaries}

Let us begin by establishing some notations.  Given two  nonnegative functions $f$ and $g$, we write $f\lesssim g$ if there exists some
positive constant $C$ such that $f\leq Cg.$ We also denote $f\approx  g$ if $f\lesssim g$ and $g\lesssim f$. If $f$
and $g$ are germ of functions on $(X, x_0 )$, we write $f\ll  g$ if $ g^{-1}(0) \subset f^{-1}(0)$ and
$\lim_{x\to x_0} [f (x)/g(x)] = 0.$

 Let $\Gamma_1, \Gamma_2$ be germs of closed subsets at $0\in\R^n$ such that $\Gamma_1\cap\Gamma_2=\{0\}$.  For $r>0$ sufficiently small, let us define
$$
f_{\Gamma_1,\Gamma_2}(r)=\inf \{\|\gamma_1-\gamma_2\| \ | \, \gamma_i\in\Gamma_i \ \mbox{and} \ \|\gamma_i \| \geq r ; \ i=1,2\}.
$$

\begin{definition}
{\rm The} contact of {\rm $\Gamma_1$ and $\Gamma_2$ is the real number below}
$$
Cont(\Gamma_1,\Gamma_2)=\lim _{r\to 0^+} \frac{\ln(f_{\Gamma_1,\Gamma_2}(r))}{\ln(r)}.
$$
\end{definition}

\begin{remark}
{\rm Notice that $Cont(\Gamma_1,\Gamma_2)$ is always at least 1 and  it may occur $Cont(\Gamma_1,\Gamma_2)=+\infty $.}
\end{remark}

\begin{proposition}\label{1}
Let $X$ and $Y$ be two germs of Euclidean closed subsets at $0$ and let $h:(X,0)\to (Y,0)$ be an $\alpha$-H\"older homeomorphism. If $\Gamma_1,\Gamma_2\subset X$ are closed, $\Gamma_1\cap\Gamma_2=\{ 0\}$  then
$$\alpha^2Cont(h(\Gamma_1),h(\Gamma_2))\leq Cont(\Gamma_1,\Gamma_2) \leq Cont(h(\Gamma_1),h(\Gamma_2))\frac{1}{\alpha^2} .$$
\end{proposition}

\begin{proof} Let $h\colon (X,0)\rightarrow (Y,0)$ be an $\alpha$-H\"older homeomorphism, in other words, for some positive constant $c$, we suppose that the homeomorphism $h$ satisfies:
$$\frac{1}{c}\| p-q \|^{\frac{1}{\alpha}} \leq \| h(p) - h(q) \| \leq c \| p- q \|^{\alpha} \ \forall p,q\in X.$$
Given $r>0$ sufficiently small, let us consider $\gamma_i\in\Gamma_i$ ($i=1,2$ ) such that
$$f_{h(\Gamma_1),h(\Gamma_2)}(r)= \| h(\gamma_1) - h(\gamma_2) \| $$ with $\|  h(\gamma_1)\|\geq r \ \mbox{and} \ \| h(\gamma_2) \| \geq r.$  Therefore,
\begin{eqnarray*}
f_{h(\Gamma_1),h(\Gamma_2)}(r) &=& \| h(\gamma_1) - h(\gamma_2) \| \\
& \geq &\frac{1}{c}\| \gamma_1 - \gamma_2\|^{\frac{1}{\alpha}} \\
& \geq & \frac{1}{c}[f_{\Gamma_1,\Gamma_2}(u)]^{\frac{1}{\alpha}} \ \mbox{where} \ cu^{\alpha}=r \\
\end{eqnarray*}
and

$$ \frac{\ln f_{h(\Gamma_1),h(\Gamma_2)}(r) }{\ln r} \leq \frac{\ln f_{\Gamma_1,\Gamma_2}(u) }{\alpha^2\ln u + \alpha\ln c} - \frac{\ln c}{\alpha\ln u + \ln c}.$$ Finally, taking $r\to 0^+$ in the last inequality, we get
$Cont(h(\Gamma_1),h(\Gamma_2))\leq \frac{1}{\alpha^2}Cont(\Gamma_1,\Gamma_2).$

In order to show that $Cont(\Gamma_1,\Gamma_2)\leq \frac{1}{\alpha^2}Cont(h(\Gamma_1),h(\Gamma_2))$, we follow a similar way using $h^{-1}$ instead $h$.

\end{proof}


\section{Plane branches}

Let $C$ be the germ of an analytically irreducible complex  curve at $0\in \mathbb{C}^{2}$ (plane branch). We know that, up to an analytic changing of coordinates, one may suppose that $C$ has a parametrization as follows:

\begin{eqnarray}\label{e1}
x &=&t^{n} \\
y &=&a_1t^{m_1}+a_{2}t^{m_{2}}+\cdots   \notag
\end{eqnarray}

where $a_1\neq 0$, $n$ is the multiplicity of $C$  and $y(t)\in\mathbb{C}\left\{t\right\}$. In the case that $0$ is a singular point of the curve, $n$ does not divide the integer number $m_1$.

The series  $y\left(x^{1/n}\right) $ with fractional exponents is known as \emph{Newton-Puiseux parametrization} of $C$ and
any other Newton-Puiseux parametrization of  $C$
is obtained from the parametrization above via $x^{1/n}\rightarrow wx^{1/n}$ where $w\in\C$ is an $n$th root of the unit.

Let us denote $\beta _{0}=n \mbox{ and }\beta _{1}=m_1$.
Let $e_{1}=gcd\hspace{0.02in}( \beta _{1},\beta _{0}) $ be the great commun divisor of these two integers. Now, we denote by $\beta _{2}$ the smaller exponent appearing in the series $y\left( t\right) $ that is not multiple of $%
e_{1} $. Let $e_{2}=gcd\hspace{0.02in}( e_{1},\beta _{2}) $; and  $%
e_{2}<e_{1}$, and so on. Let us suppose that we have defined  $e_{i}$ $%
=gcd\hspace{0.02in}( e_{i-1},\beta _{i}) $. Thus, we define $\beta _{i+1}$ as the smaller exponent of the series $y\left( t\right) $ that is not multiple of $e_{i} $. Since the sequence of positive integers
\begin{equation*}
n>e_{1}>\cdots >e_{i}>\cdots
\end{equation*}%
is decreasing, there exists an integer number $g$ such that $e_{g}=1$ . In this way, we can rewrite Eq. \ref{e1} as follows:%
\begin{eqnarray*}
x &=&t^{n} \\
y &=&a_{\beta_1}t^{\beta _{1}}+a_{\beta _{1}+e_{1}}t^{\beta _{1}+e_{1}}+\cdots +a_{\beta
_{1}+k_{1}e_{1}}t^{\beta _{1}+k_{1}e_{1}} \\
&&+a_{\beta _{2}}t^{\beta _{2}}+a_{\beta _{2}+e_{2}}t^{\beta
_{2}+e_{2}}+\cdots +a_{\beta _{q}}t^{\beta _{q}}+a_{\beta
_{q}+e_{q}}t^{\beta _{q}+e_{q}}+\cdots \\
&&+a_{\beta _{g}}t^{\beta _{g}}+a_{\beta _{g}+1}t^{\beta _{g}+1}+\cdots
\end{eqnarray*}%
where the coefficient of $t^{\beta _{i}}$ is nonzero ($1\leq i\leq g$). Now, we define the integers  $m_{i}$ and $n_i$ via the following equations:
\begin{eqnarray*}
e_{i-1} &=&n_{i}e_{i} \\
\beta _{i} &=&m_{i}e_{i}
\end{eqnarray*}%
Thus, one may expand $y$ as a fractional power series of $x$ in the following way:%
\begin{eqnarray*}
y\left( x^{1/n}\right)  &=&a_{\beta _{1}}x^{\tfrac{m_{1}}{n_{1}}}+a_{\beta
_{1}+e_{1}}x^{\tfrac{m_{1}+1}{n_{1}}}+\cdots +a_{\beta _{1}+k_{1}e_{1}}x^{%
\tfrac{m_{1}+k_{1}}{n_{1}}} \\
&&+a_{\beta _{2}}x^{\tfrac{m_{2}}{n_{1}n_{2}}}+a_{\beta _{2}+e_{2}}x^{\tfrac{%
m_{2}+1}{n_{1}n_{2}}}+\cdots +a_{\beta _{q}}x^{\tfrac{m_{q}}{%
n_{1}n_{2}\ldots n_{q}}}+a_{\beta _{q}+e_{q}}x^{\tfrac{m_{q}+1}{%
n_{1}n_{2}\ldots n_{q}}}+\cdots  \\
&&+a_{\beta _{g}}x^{\tfrac{m_{g}}{n_{1}n_{2}\ldots n_{g}}}+a_{\beta
_{g}+e_{g}}x^{\tfrac{m_{g}+1}{n_{1}n_{2}\ldots n_{g}}}+\cdots
\end{eqnarray*}%
The sequence of integers $\left( \beta_{0}, \beta_{1},\beta _{2},\ldots ,\beta _{g}\right) $ is called  the \emph{characteristic exponents} of  $(C,0)$, and the sequence $%
\left( m_{1},n_{1}\right), \ldots, \left( m_{g},n_{g}\right) $ \ is called the \emph{characteristic pairs} of  $(C,0)$.

\begin{remark}\label{remark2.4}
\emph {Any plane branch $C$  with characteristic exponents $\left(\beta_{0}, \beta
_{1},\beta _{2},\ldots ,\beta _{g}\right) $, is bi-Lipschitz homeomorphic to another analytic plane branch parametrized in the following way:
\begin{eqnarray*}
y\left( x^{1/n}\right) &=&a_{\beta_{1}}x^{\tfrac{m_{1}}{n_{1}}}+
a_{\beta _{2}}x^{\tfrac{m_{2}}{n_{1}n_{2}}}+\cdots
+a_{\beta _{g}}x^{\tfrac{m_{g}}{n_{1}n_{2}\ldots n_{g}}}
\end{eqnarray*}%
}
\end{remark}

\section{Main results}

Let us begin this section by stating one of the main results of the paper.

\begin{theorem}\label{main-theorem}
Let $C$ and $\tilde{C}$ be complex analytic plane branches. If $C$ and $\tilde{C}$ have different sequence of characteristic exponents, then there exists $0<\alpha<1$ such that $C$ is not $\alpha$-H\"older homeomorphic to $\tilde{C}$. In particular, the branches are not  Lipschitz homeomorphic.
\end{theorem}

The next example give us an idea  how to get a proof of Theorem \ref{main-theorem}.
\begin{example}
Let $ C: y^{2} = x^{5}$ and $\tilde{C}: y^{2} = x^{3}.$ There is no a bi-$\frac{4}{5}$-H\"{o}lder homeomorphism $F\colon\left( C,0\right) \rightarrow (\tilde{C},0).$
\end{example}

\begin{proof} Let us suppose that there is a  bi-$\frac{4}{5}$-H\"{o}lder homeomorphism $F\colon\left( C,0\right) \rightarrow(\tilde{C},0)$. Let $\Sigma _{1},\Sigma _{2},\Sigma _{3},\Sigma _{4}
$ be the following real arcs in  $(C,0)$:%
\begin{equation*}
\Sigma _{1}=\{(r,r^{5/2}):r\geq 0\}; \ \ \ \ \ \ \ \ \ \ \ \ \ \ \ \
\end{equation*}%
\begin{equation*}
\Sigma _{2}=\{(ri,r^{5/2}e^{i\left( 5/2\right) \pi /2}):r\geq 0\};\ \ \ \
\end{equation*}%
\begin{equation*}
\Sigma _{3}=\{(r,-r^{5/2}):r\geq 0\};\ \ \ \ \ \ \ \ \ \ \ \ \ \
\end{equation*}%
\begin{equation*}
\Sigma _{4}=\{(-ri,r^{5/2}e^{i\left( 5/2\right) 11\pi /2}):r\geq 0\},
\end{equation*}%
Let us define \textbf{:}%
\begin{equation*}
\Gamma _{k}\left( r\right) =(re^{i\gamma _{k}\left( r\right)
},r^{3/2}e^{i\left( 3/2\right) \gamma _{k}\left( r\right)
})\in F\left( \Sigma _{k}\right) ;\text{ }k=1,2,3,4
\end{equation*}%
and,
\begin{equation*}
r_{k,n}=\|F(\Sigma_k(r_n^{\frac{1}{\alpha}}))\|, \mbox{ para } k=1,2 \mbox{ and }3.
\end{equation*}

It comes from Proposition \ref{1}
\begin{equation*}
\left\Vert \Gamma _{1}\left( r_{1,n}\right) -\Gamma _{3}\left( r_{3,n}\right)
\right\Vert \lesssim( r_n)^{2\alpha^2}.
\end{equation*}%

We know that either
$$|\gamma _{1}\left( r\right) -\gamma _{2}\left( r\right) |\leq |\gamma _{1}\left( r\right) -\gamma _{3}\left( r\right) |$$
or
$$|\gamma _{1}\left( r\right) -\gamma _{4}\left( r\right) |\leq |\gamma _{1}\left( r\right) -\gamma _{3}\left( r\right) |,$$
for any $r>0$. Thus, up to subsequences, one can suppose, for instance, that
\begin{equation*}
|\gamma _{1}\left( r_{1,n}\right) -\gamma _{2}\left( r_{3,n}\right) |\leq |\gamma _{1}\left( r_{1,n}\right) -\gamma _{3}\left( r_{3,n}\right) |, \forall n,
\end{equation*}
hence,
\begin{equation*}
\left\Vert \Gamma_{1}\left( r_{1,n} \right) -\Gamma_{2}\left( r_{3,n} \right)
\right\Vert \leq \left\Vert \Gamma_{1}\left( r_{1,n} \right) -\Gamma_{3}\left( r_{3,n} \right)
\right\Vert ,\forall n.
\end{equation*}%

By denoting $\delta _{j}(r_{k,n}\mathbf{)}=F^{-1}\left( \Gamma _{j}\left( r_{k,n}\right)
\right) ;j,k=1,2,3$, we get
\begin{equation*}
\delta _{1}(r_{1,n}\mathbf{)=(}h\left( r_{1,n}\right) ,h\left( r_{1,n}\right) ^{5/2})
\end{equation*}
and
\begin{equation*}
\delta _{2}(r_{2,n}\mathbf{)=(}%
ig\left( r_{2,n}\right) ,g\left( r_{2,n}\right) ^{5/2}e^{i\left( 5/2\right) \pi
/2}),
\end{equation*}%
where $(r_n)^{\tfrac{1}{\alpha }}\lesssim \left\vert h\left( r_{1,n}\right) \right\vert
\approx \left\vert g\left( r_{2,n}\right) \right\vert \lesssim (r_n)^{\alpha }.$
Hence,%
\begin{equation*}
\left\vert \delta _{1}\left( r_{1,n}\right) -\delta _{2}\left( r_{2,n}\right)
\right\vert \gtrsim \left\vert h\left( r_{1,n}\right) -ig\left( r_{2,n}\right)
\right\vert \gtrsim (r_n)^{\tfrac{1}{\alpha }}\ .
\end{equation*}%
Therefore,
\begin{eqnarray*}
(r_n)^{\tfrac{1}{\alpha^2 }} &\lesssim &  \left\Vert \delta _{1}\left( r_{1,n}\right)
-\delta _{2}\left( r_{2,n}\right) \right\Vert ^{\tfrac{1}{\alpha }} = (f_{\delta_1,\delta_3}(r))^{\tfrac{1}{\alpha}} \leq \left\Vert
\delta _{1}\left( r_{1,n}\right) -\delta _{2}\left( r_{3,n}\right) \right\Vert ^{\tfrac{1}{\alpha }} \\ &=& \left\Vert
F^{-1}\left( \Gamma _{1}\left( r_{1,n}\right)\right)) - F^{-1}\left(\Gamma _{2}\left( r_{3,n}\right)\right) \right\Vert ^{\tfrac{1}{\alpha }} \leq \left\Vert
\Gamma _{1}\left( r_{1,n}\right) -\Gamma _{2}\left( r_{3,n}\right) \right\Vert \\  &\leq &  \left\Vert
\Gamma _{1}\left( r_{1,n}\right) -\Gamma _{3}\left( r_{3,n}\right) \right\Vert \lesssim
(r_n)^{2\alpha^2 }
\end{eqnarray*}%
and, this implies $ 2\alpha^2\leq \frac{1}{\alpha^2}.$ Then, $\alpha^4 \leq \frac{1}{2}< \frac{4}{5}$, what is a contradiction.

The other cases are analyzed in a completely similar way.
\end{proof}

In the following, we are going to generalize what was proved in the example above. Let  $C$ and $\tilde{C}$ be branches of complex analytic plane curves at $0\in\C^2$ with the following characteristic pairs
$\left( n_{1},m_{1}\right), \left( n_{2},m_{2}\right) ,\ldots ,\left(n_{g},m_{g}\right)$ and $\left(q_{1},l_{1}\right)$, $\left( q_{2},l_{2}\right)$,$\ldots$ ,$\left( q_{\tilde{g}},l_{\tilde{g}}\right)$ respectively. Before the next result, let us define the following rational number
\begin{equation}\label{kij}
k_{ij}(C,\tilde C)=\min \{\dfrac{m_{j}.q_{1}\ldots q_{i}+l_{i}.n_{1}.\ldots n_{j}}{2.l_{i}.n_{1}.\ldots n_{j}},\dfrac{m_{j}.q_{1}\ldots q_{i}+l_{i}.n_{1}.\ldots n_{j}}{2.m_{j}.q_{1}\ldots q_{i}}\}.
\end{equation}

\begin{lemma}
\label{lemma3.2} If $g\neq \tilde{g}$ and $\alpha _{0}$ is a positive real number such that
 $$\max \{k_{ij}(C,\tilde{C});i=\tilde{g}\leq j\leq g\hspace{0.1in}\mbox{or}\hspace{0.1in}j=g\leq i\leq \tilde{g}\}<\alpha _{0}^{4}<1,$$
 then there is not any bi-$\alpha $-H\"{o}lder homeomorphism $F\colon\left( C,0\right) \rightarrow (\tilde{C},0),$ with $\alpha_{0}<\alpha <1.$
\end{lemma}

\begin{proof} Without loss of generality, we can suppose that $(C,0)$ and $(\tilde{C},0)$  are parametrized as in Remark \ref{remark2.4} and let us suppose that $g>\tilde{g}.$ In this way, we have the following three cases:

\begin{enumerate}
\item $\dfrac{l_{\tilde{g}}}{\tilde{n}}\leq \dfrac{m_{j}}{n_{1}.\ldots n_{j}},$
$\forall \tilde{g}\leq j\leq g;$

\item $\dfrac{m_{j}}{n_{1}.\ldots n_{j}}\leq \dfrac{l_{\tilde{g}}}{\tilde{n}},$
$\forall \tilde{g}\leq j\leq g;$

\item $\exists $ $j_{0},$ $\tilde{g}\leq j_{0}\leq g$ such that $\dfrac{m_{j}}{%
n_{1}.\ldots n_{j}}\leq \dfrac{l_{\tilde{g}}}{\tilde{n}},\forall \tilde{g}%
\leq j\leq j_{0} \mbox{ and } \dfrac{l_{\tilde{g}}}{%
\tilde{n}}<\dfrac{m_{j}}{n_{1}.\ldots n_{j}},$ $\forall j_{0}<j\leq g.$
\end{enumerate}

Notice that for any of the above cases, there exists at most one index $j$ such tat $\dfrac{l_{%
\tilde{g}}}{\tilde{n}}=\dfrac{m_{j}}{n_{1}.\ldots n_{j}}$. So, we are going to consider just $j$ such that $\dfrac{l_{\tilde{g}}}{\tilde{n}%
}\neq \dfrac{m_{j}}{n_{1}.\ldots n_{j}}.$ For instance, let us suppose that $\dfrac{l_{\tilde{g}}}{%
\tilde{n}}>\dfrac{m_{j}}{n_{1}.\ldots n_{j}},$ $\tilde{g}\leq j\leq g.$ In this case, let us consider $\Gamma_{1},\Gamma_{2},\Gamma_{3},\Gamma_{4}$ the following arcs on  $(\tilde{C},0) :$%
\begin{equation*}
\begin{array}{c}
\Gamma_{1}=\{(r,b_{\beta _{1}}r^{\tfrac{l_{1}}{q_{1}}}+b_{\beta _{2}}r^{\tfrac{l_{2}}{q_{1}q_{2}}}+\cdots \text{}
+\,b_{\beta _{g}}r^{\tfrac{l_{\tilde{g}}}{\tilde{n}}} ):r\geq 0\};\ \ \ \ \ \ \ \ \ \ \ \ \ \ \ \ \ \ \ \ \ \ \ \ \ \ \ \ \ \ \ \ \ \ \ \ \ \
\end{array}%
\end{equation*}%
\begin{equation*}
\begin{array}{c}
\Gamma_{2}=\{(ri,b_{\beta _{1}}r^{\tfrac{l_{1}}{q_{1}}}e^{i\left( \tfrac{l_{1}}{q_{1}}\right)\tfrac{\pi}{2}}+\,b_{\beta _{2}}r^{\tfrac{l_{2}}{q_{1}q_{2}}}e^{i\left( \tfrac{l_{2}}{q_{1}q_{2}}\right) \tfrac{\pi}{2}}
+\cdots +\,b_{\beta _{\tilde{g}}}r^{\tfrac{l_{\tilde{g}}}{\tilde{n}}}e^{i\left( \tfrac{l_{\tilde{g}}}{\tilde{n}}\right) \tfrac{\pi}{2}}):r\geq 0\};\ \ \ \ \
\end{array}%
\end{equation*}%
\begin{equation*}
\begin{array}{c}
\Gamma_{3}=\{(r,b_{\beta _{1}}r^{\tfrac{l_{1}}{q_{1}}}+b_{\beta
_{2}}r^{\tfrac{l_{2}}{q_{1}q_{2}}}+\cdots +\,b_{\beta _{\tilde{g}-1}}r^{\tfrac{l_{\tilde{g}-1}}{q_{1}\ldots q_{\tilde{g}%
-1}}}
+\,b_{\beta _{\tilde{g}}}r^{\tfrac{l_{\tilde{g}}}{\tilde{n}}}e^{i\left( \tfrac{l_{\tilde{g}}}{\tilde{n}}\right) 2q_{1}\ldots q_{\tilde{g}-1}}):r\geq 0\};%
\end{array}%
\end{equation*}%
\begin{equation*}
\begin{array}{c}
\Gamma_{4}=\{(-ri,b_{\beta _{1}}r^{\tfrac{l_{1}}{q_{1}}}e^{i\left(\tfrac{l_{1}}{q_{1}}\right) \left( \tfrac{4q_{1}...q_{\tilde{g}-1}+3}{2}\right)\pi } +\,b_{\beta _{2}}r^{\tfrac{l_{2}}{q_{1}q_{2}}}e^{i\left( \tfrac{l_{2}}{q_{1}q_{2}}\right) \left( \tfrac{4q_{1}...q_{\tilde{g}-1}+3}{2}\right) \pi }+\cdots \text{ \ \ \ \ \ }\\
+\,b_{\beta _{\tilde{g}}}r^{\tfrac{l_{\tilde{g}}}{\tilde{n}}}e^{i\left( \tfrac{l_{\tilde{g}}}{\tilde{n}}\right)  \left( \tfrac{4q_{1}...q_{\tilde{g}-1}+3}{2}\right)\pi } ):r\geq 0\}. \ \ \ \ \ \ \ \ \ \ \ \ \ \ \ \ \ \ \ \ \ \ \ \ \ \ \ \ \ \ \ \ \
\end{array}%
\end{equation*}%
Let
\begin{equation*}
\begin{array}{c}
\Sigma_{k}\left( r\right) =(re^{i\sigma_{k}\left( r\right) },a_{\beta_{1}}r^{\tfrac{m_{1}}{n_{1}}}e^{i\left( \tfrac{m_{1}}{n_{1}}\right) \sigma_{k}\left( r\right)} +a_{\beta _{2}}r^{\tfrac{m_{2}}{n_{1}n_{2}}}e^{i\left( \tfrac{m_{2}}{n_{1}n_{2}}\right) \sigma_{k}\left( r\right) }
+\cdots \text{ \ \ \ \ \ \ \ \ \ \ \ \ \ \ \ \ \ \ \ \ \ \ \ \ \ }\\
\ \ \ \ \ \ \ \ \ \ +a_{\beta _{g}}r^{\tfrac{m_{g}}{n}}e^{i\left( \tfrac{m_{g}}{n}\right) \sigma_{k}\left(
r\right) } )\in F^{-1}\left( \Gamma_{k}\right) ;\text{ }k=1,2,3,4.\text{ \ \ \ \ \ \ \ \ \ \ \ \ \ \ \ \ }%
\end{array}%
\end{equation*}%

At this moment, let us take $\alpha_{0}<\alpha <1$ and suppose that there exists a  bi-$\alpha $-H\"{o}lder homeomorphism
$F\colon\left( C,0\right) \rightarrow (\tilde{C},0).$

Let us define
\begin{equation*}
r_{k,n}=\|F^{-1}(\Gamma_k(r_n^{\frac{1}{\alpha}}))\|, \mbox{ for } k=1,2 \mbox{ and } 3.
\end{equation*}

It comes from Proposition \ref{1} that:%

\begin{equation*}
\left\Vert
\Sigma _{1}\left( r_{1,n}\right) -\Sigma _{3}\left( r_{3,n}\right) \right\Vert \lesssim (r_n)^{\dfrac{l_{\tilde{g}}.n_{1}.\ldots n_{j}+m_{j}.\tilde{n}}{2.\tilde{n}.n_{1}.\ldots n_{j}}\alpha^2}
\end{equation*}%

Moreover, we know that either%
$$|\sigma_{1}\left( r\right) -\sigma_{2}\left( r\right) |\leq |\sigma_{1}\left( r\right) -\sigma_{3}\left( r\right) |$$
or
$$|\sigma_{1}\left( r\right) -\sigma_{4}\left( r\right) |\leq |\sigma_{1}\left( r\right) -\sigma_{3}\left( r\right) |,$$
$\forall \ r>0$. Up to a subsequence, we may suppose that
\begin{equation*}
|\sigma_{1}\left( r_{1,n}\right) -\sigma_{2}\left( r_{3,n}\right) |\leq |\sigma_{1}\left( r_{1,n}\right) -\sigma_{3}\left( r_{3,n}\right) |, \forall n.
\end{equation*}
$\therefore$
\begin{equation*}
\left\Vert \Sigma _{1}\left( r_{1,n}\right) -\Sigma _{2}\left( r_{3,n}\right)
\right\Vert \leq \left\Vert \Sigma _{1}\left( r_{1,n}\right) -\Sigma _{3}\left( r_{3,n}\right)
\right\Vert ,\forall n.
\end{equation*}%

Now, let us denote  $\delta _{j}(r_{k,n}\mathbf{)}=F\left( \Sigma_{j}\left( r_{k,n}\right)
\right) ;j,k=1,2,3$. Thus,
\begin{equation*}
\delta _{1}(r_{1,n}\mathbf{)=(}h\left( r_{1,n}\right) ,b_{\beta _{1}}h\left( r_{1,n}\right) ^{l_{1}/q_{1}}+\ldots) \
\end{equation*}
and
\begin{equation*}
\delta _{2}(r_{2,n}\mathbf{)=(}%
ig\left( r_{2,n}\right) ,b_{\beta_{1}}g\left( r_{2,n}\right) ^{l_{1}/q_{1}}e^{i\left( l_{1}/q_{1}\right) \pi
/2}+\ldots)
\end{equation*}%
where $$(r_n)^{\tfrac{1}{\alpha }}\lesssim \left\vert h\left( r_{1,n}\right) \right\vert
\approx \left\vert g\left( r_{2,n}\right) \right\vert \lesssim (r_n)^{\alpha }$$
$\therefore$%
$$
\left\vert \delta _{1}\left( r_{1,n}\right) -\delta _{2}\left( r_{2,n}\right)
\right\vert \gtrsim \left\vert h\left( r_{1,n}\right) -ig\left( r_{2,n}\right)
\right\vert \gtrsim (r_n)^{\tfrac{1}{\alpha }}\ .
$$
Hence,
\begin{eqnarray*}
(r_n)^{\tfrac{1}{\alpha^2 }} &\lesssim & \left\Vert \delta _{1}\left(r_{1,n} \right)-\delta _{2}\left( r_{2,n}\right) \right\Vert ^{\tfrac{1}{\alpha }} = (f_{\delta_1,\delta_2}(r))^{\tfrac{1}{\alpha}} \leq \left\Vert
\delta _{1}\left( r_{1,n}\right) -\delta _{2}\left( r_{3,n}\right) \right\Vert ^{\tfrac{1}{\alpha }} \\ &=& \left\Vert
F\left( \Sigma _{1}\left( r_{1,n}\right)\right)) - F\left(\Sigma _{2}\left( r_{3,n}\right)\right) \right\Vert ^{\tfrac{1}{\alpha }} \lesssim \left\Vert
\Sigma_{1}\left( r_{1,n}\right) -\Sigma_{2}\left( r_{3,n}\right) \right\Vert \\ &\leq &  \left\Vert
\Sigma _{1}\left( r_{1,n}\right) -\Sigma _{3}\left( r_{3,n}\right) \right\Vert  \lesssim
(r_n)^{\dfrac{l_{\tilde{g}}.n_{1}.\ldots n_{j}+m_{j}.\tilde{n}}{2.\tilde{n}.n_{1}.\ldots n_{j}}\alpha^2 }
\end{eqnarray*}%
and, this implies $$\dfrac{l_{\tilde{g}}.n_{1}.\ldots n_{j}+m_{j}.\tilde{n}}{2.\tilde{n}.n_{1}.\ldots n_{j}}\alpha^2\leq \frac{1}{\alpha^2}.$$ Then, $${\alpha^4}\leq \frac{2.\tilde{n}.n_{1}.\ldots n_{j}}{l_{\tilde{g}}.n_{1}.\ldots n_{j}+m_{j}.\tilde{n}}< \frac{l_{\tilde{g}}.n_{1}.\ldots n_{j}+m_{j}.\tilde{n}}{2.l_{\tilde{g}}.n_{1}.\ldots n_{j}},$$ what is a contradiction.

The other cases are analyzed in a completely similar way.
\end{proof}

\begin{lemma}
\label{lemma3.0}Let $(C,0)$ e $(\tilde{C},0)$ be two complex analytic plane branches with $g=\tilde{g}$. If $1>\alpha _{0}^{4}> k_{ii}(C,\tilde{C})\neq 1$, $1\leq i\leq g $ ,  then there is no  bi-$
\alpha $-H\"{o}lder homeomorphism $F\colon\left( C,0\right) \rightarrow (\tilde{C},0),$ $%
\forall $ $\alpha _{0}<\alpha <1.$
\end{lemma}

\begin{proof} Let $(C,0)$ and $(\tilde{C},0)$ be parametrized as in Remark \ref{remark2.4}. Let $1\leq j\leq g$ be such that
$\dfrac{l_{j}n_{1}\ldots n_{j}}{m_{j}q_{1}\ldots q_{j}}\neq 1.$ Let us suppose that $\dfrac{l_{j}n_{1}\ldots n_{j}}{m_{j}q_{1}\ldots q_{j}}<1$, that is
$\dfrac{l_{j}}{q_{1}\ldots q_{j}}<\dfrac{m_{j}}{n_{1}\ldots n_{j}}.$ Let us consider
$\Sigma_{1},\Sigma_{2},\Sigma_{3},\Sigma_{4}$ the following arcs in $C$:%
\begin{equation*}
\begin{array}{c}
\Sigma_{1}=\{(r,a_{\beta _{1}}r^{\tfrac{m_{1}}{n_{1}}}+\,a_{\beta _{2}}r^{\tfrac{m_{2}}{n_{1}n_{2}}}+\cdots
+\,a_{\beta_{j}}r^{\tfrac{m_{j}}{n_{1}\ldots n_{j}}}+\cdots  ):r\geq 0\};
\end{array}%
\end{equation*}%
\begin{equation*}
\begin{array}{c}
\Sigma_{2}=\{(ri,a_{\beta _{1}}r^{\tfrac{m_{1}}{n_{1}}}e^{i\left( \tfrac{m_{1}}{n_{1}}\right)\tfrac{\pi }{2}}+
\cdots +\,a_{\beta _{j}}r^{\tfrac{m_{j}}{n_{1}\ldots n_{j}}}e^{i\left( \tfrac{m_{j}}{n_{1}\ldots n_{j}}\right) \tfrac{\pi }{2}}+\cdots  ):r\geq 0\};
\end{array}%
\end{equation*}%
\begin{equation*}
\begin{array}{c}
\Sigma_{3}=\{(r,a_{\beta _{1}}r^{\tfrac{m_{1}}{n_{1}}}+\cdots +\,a_{\beta _{j-1}}r^{\tfrac{m_{j-1}}{n_{1}\ldots n_{j-1}}}
+\,a_{\beta _{j}}r^{\tfrac{m_{j}}{n_{1}\ldots n_{j}}}e^{i\left( \tfrac{m_{j}}{n_{1}\ldots n_{j}}\right) 2n_{1}\ldots n_{j-1}}+\cdots ):r\geq 0\};
\end{array}%
\end{equation*}%
\begin{equation*}
\begin{array}{c}
\Sigma_{4}=\{(-ri,a_{\beta _{1}}r^{\tfrac{m_{1}}{n_{1}}}e^{i\left(\tfrac{m_{1}}{n_{1}}\right)\left(\tfrac{4n_{1}...n_{g-1}+3}{2}\right) \pi}+\,a_{\beta _{2}}r^{\tfrac{m_{2}}{n_{1}n_{2}}}e^{i\left( \tfrac{m_{2}}{n_{1}n_{2}}\right) \left(\tfrac{4n_{1}...n_{g-1}+3}{2}\right) \pi }+\cdots \text{ \ \ \ \ \ \ \ \ \ }\\
+\,a_{\beta _{j}}r^{\tfrac{m_{j}}{n_{1}\ldots n_{j}}}e^{i\left( \tfrac{m_{j}}{n_{1}\ldots n_{j}}\right)\left( \tfrac{4n_{1}...n_{j-1}+3}{2}\right) \pi }+\cdots):r\geq 0\}. \ \ \ \ \ \ \ \ \ \ \ \ \ \ \ \ \
\end{array}%
\end{equation*}%
Let
\begin{equation*}
\begin{array}{c}
\Gamma_{k}\left( r\right) =(re^{i\sigma _{k}\left( r\right) },b_{\beta_{1}}r^{\tfrac{l_{1}}{q_{1}}}e^{i\left( \tfrac{l_{1}}{q_{1}}\right) \sigma _{k}\left( r\right)}+\,b_{\beta _{2}}r^{\tfrac{l_{2}}{q_{1}q_{2}}}e^{i\left( \tfrac{l_{2}}{q_{1}q_{2}}\right) \sigma_{k}\left( r\right) }
+\cdots \text{ \ \ \ \ \ \ \ \ \ \ \ \ \ \ \ \ \ \ \ \ \ \ \ \ \ \ \ }\\
+\,b_{\beta _{j}}r^{\tfrac{l_{j}}{q_{1}\ldots q_{j}}}e^{i\left( \tfrac{l_{j}}{q_{1}\ldots q_{j}}\right) \sigma _{k}\left( r\right) }+\cdots )\in F\left( \Sigma_{k}\right) ;\text{ }k=1,2,3,4.
\end{array}%
\end{equation*}%

By contradiction, let us suppose that there exists a bi-$\alpha $-H\"{o}lder homeomorphism
$F\colon\left( C,0\right) \rightarrow (\tilde{C},0),$ where $\alpha_{0}<\alpha <1.$

In this way, for each $n$, we define $r_n>0$ satisfying

\begin{equation*}
r_{k,n}=\|F^{-1}(\Gamma_k(r_n^{\frac{1}{\alpha}}))\|, \mbox{ for } k=1,2 \mbox{ and } 3.
\end{equation*}

It comes from Proposition \ref{1} that%

\begin{equation*}
\left\Vert \Gamma_{1}\left( r_{1,n}\right) -\Gamma_{3}\left( r_{3,n}\right)
\right\Vert \lesssim (r_n)^{\dfrac{l_{j}.n_{1}.\ldots n_{j}+m_{j}.q_{1}.\ldots.q_{j}}{2.q_{1}.\ldots.q_{j}.n_{1}.\ldots n_{j}}\alpha^2}
\end{equation*}%

Moreover, we know that either%
$$|\gamma_{1}\left( r\right) -\gamma_{2}\left( r\right) |\leq |\gamma_{1}\left( r\right) -\gamma_{3}\left( r\right) |$$
or
$$|\gamma_{1}\left( r\right) -\gamma_{4}\left( r\right) |\leq |\gamma_{1}\left( r\right) -\gamma_{3}\left( r\right) |,$$
for all  $r>0$. Up to a subsequence, we may suppose that
\begin{equation*}
|\gamma_{1}\left( r_{1,n}\right) -\gamma_{2}\left( r_{3,n}\right) |\leq |\gamma_{1}\left( r_{1,n}\right) -\gamma_{3}\left( r_{3,n}\right) |, \forall n.
\end{equation*}
$\therefore$
\begin{equation*}
\left\Vert \Gamma_{1}\left( r_{1,n}\right) -\Gamma_{2}\left( r_{3,n}\right)
\right\Vert \leq \left\Vert \Gamma_{1}\left( r_{1,n}\right) -\Gamma_{3}\left( r_{3,n}\right)
\right\Vert ,\forall n.
\end{equation*}%

Let us denote $\delta _{j}(r_{k,n}\mathbf{)}=F^{-1}\left( \Gamma _{j}\left( r_{k,n}\right)
\right) ;j,k=1,2,3$. Thus,
$$
\delta _{1}(r_{1,n}\mathbf{)=(}h\left( r_{1,n}\right) ,a_{\beta _{1}}h\left( r_{1,n}\right)^{m_{1}/n_{1}}+\ldots)
$$ and
$$\delta _{2}(r_{2,n}\mathbf{)=(}%
ig\left( r_{2,n}\right) ,a_{\beta_{1}}g\left( r_{2,n}\right) ^{m_{1}/n_{1}}e^{i\left( m_{1}/n_{1}\right) \pi
/2}+\ldots)
$$
where $(r_n)^{\tfrac{1}{\alpha }}\lesssim \left\vert h\left( r_{1,n}\right) \right\vert
\approx \left\vert g\left( r_{2,n}\right) \right\vert \lesssim (r_n)^{\alpha }.$
Therefore,%
\begin{equation*}
\left\vert \delta _{1}\left( r_{1,n}\right) -\delta _{2}\left( r_{2,n}\right)
\right\vert \gtrsim \left\vert h\left( r_{1,n}\right) -ig\left( r_{2,n}\right)
\right\vert \gtrsim (r_n)^{\tfrac{1}{\alpha }}\ .
\end{equation*}%
Hence,
\begin{eqnarray*}
(r_n)^{\tfrac{1}{\alpha^2 }} &\lesssim& \left\Vert \delta _{1}\left( r_{1,n}\right)
-\delta _{2}\left( r_{2,n}\right) \right\Vert ^{\tfrac{1}{\alpha }} = (f_{\delta_1,\delta_2}(r))^{\tfrac{1}{\alpha}} \leq \left\Vert
\delta _{1}\left( r_{1,n}\right) -\delta _{2}\left( r_{3,n}\right) \right\Vert ^{\tfrac{1}{\alpha }} \\ &=& \left\Vert
F^{-1}\left( \Gamma _{1}\left( r_{1,n}\right)\right)) - F^{-1}\left(\Gamma _{2}\left( r_{3,n}\right)\right) \right\Vert ^{\tfrac{1}{\alpha }} \lesssim \left\Vert
\Gamma_{1}\left( r_{1,n}\right) -\Gamma_{2}\left( r_{3,n}\right) \right\Vert \\ &\leq &  \left\Vert
\Gamma _{1}\left( r_{1,n}\right) -\Gamma _{3}\left( r_{3,n}\right) \right\Vert  \lesssim
(r_n)^{\dfrac{l_{j}.n_{1}.\ldots n_{j}+m_{j}.q_{1}.\ldots.q_{j}}{2.q_{1}.\ldots.q_{j}.n_{1}.\ldots n_{j}}\alpha^2 }
\end{eqnarray*}%
and, therefore,
$$
\dfrac{l_{j}.n_{1}.\ldots n_{j}+m_{j}.q_{1}.\ldots.q_{j}}{2.q_{1}.\ldots.q_{j}.n_{1}.\ldots n_{j}}\alpha^2  \leq  \frac{1}{\alpha^2},$$ that is
$${\alpha^4}\leq \frac{2.q_{1}.\ldots.q_{j}.n_{1}.\ldots n_{j}}{l_{j}.n_{1}.\ldots n_{j}+m_{j}.q_{1}.\ldots.q_{j}}< \frac{l_{j}.n_{1}.\ldots n_{j}+m_{j}.q_{1}.\ldots.q_{j}}{2.m_{j}.q_{1}.\ldots q_{j}},
$$
what is a contradiction.

The other cases are similar.
\end{proof}

\begin{proof}[Proof of Theorem \ref{main-theorem}]

Let us suppose by contradiction that $C$ and $\tilde{C}$ are $\alpha$-H\"older homeomorphic for all $\alpha\in (0,1)$. From Lemma \ref{lemma3.2}, we get $g=\tilde{g}$, and by Lemma \ref{lemma3.0}, we know that  $$\dfrac{l_{i}}{q_{1}\ldots q_{i}}=\dfrac{m_{i}}{n_{1}\ldots n_{i}},\forall i.$$

By taking $i=1$,  in the previous equation, we get
$\dfrac{m_{1}}{n_{1}}=\dfrac{l_{1}}{q_{1}}$, that is, $n_{1}=q_{1}$  and $m_{1}=l_{1}.$ By taking $i=2$, in the previous equation, we get $$\dfrac{m_{2}}{n_{1}n_{2}}=\dfrac{l_{2}}{q_{1}q_{2}}.$$ Since $n_{1}=q_{1},$ it follows that $n_{2}=q_{2}$ and $m_{2}=l_{2}.$

Following in that way, for $i=g,$ we  get  $$\dfrac{m_{g}}{n_{1}\ldots
n_{g}}=\dfrac{l_{g}}{q_{1}\ldots q_{g}}.$$
Since we have proved that $n_{1}=q_{1},n_{2}=q_{2},
\cdots ,n_{g-1}=q_{g-1},$ we  have $n_{g}=q_{g}$ and $m_{g}=l_{g}.$ Then $(m_{1},n_{1})=(l_{1},q_{1}),$ $%
(m_{2},n_{2})=(l_{2},q_{2}),...,(m_{g},n_{g})=(l_{g},q_{g}),$ hence $(C,0) $ and $(\tilde{C},0)$  have the same characteristic exponents.
\end{proof}


In the next, we are dealing with germs of complex analytic plane curves having more than one branch at $0\in\C^2$ and we are going to arrive in a result like Theorem \ref{main-theorem}. Let us start pointing out the following version of Proposition \ref{1} for  germs of complex analytic plane curves with several branches.

\begin{proposition}\label{contact2}
Let $C$ and $\tilde{C}$ be germs of complex analytic plane curves at $0\in\C^2$. Let $h:(C,0)\to (\tilde{C},0)$ be a bi-$\alpha$-H\"older homeomorphism. If $C_1,...,C_r$ are the irreducible components of $C$, then $h(C_1),...,h(C_r)$ are the irreducible components of $\tilde{C}$ and
$$\alpha^2\leq \frac{Cont(C_i,C_j)}{Cont(h(C_i),h(C_j))}\leq \frac{1}{\alpha^2} .$$
\end{proposition}

\begin{proof}
By  Lemma A.8 in \cite{Gau-Lipman:1983}, it follows that  $h(C_1),...,h(C_r)$ are the irreducible components of $\tilde{C}$ and, by Proposition \ref{1}, $$\alpha^2\leq \frac{Cont(C_i,C_j)}{Cont(h(C_i),h(C_j))}\leq \frac{1}{\alpha^2} .$$
\end{proof}

\begin{theorem}\label{main-result}
Let $C_1,C_2,\tilde{C}_1$ and $\tilde{C}_2$ be complex analytic plane branches. If $Cont(C_1,C_2)\not=Cont(\tilde{C}_1,\tilde{C}_2)$, then there exists $0<\alpha<1$ such that $C=C_1\cup C_2$ is not $\alpha$-H\"older homeomorphic to $\tilde{C}=\tilde{C}_1\cup\tilde{C}_2$.
\end{theorem}
\begin{proof}
Let us take $$\alpha_0^2=\min\left\{\frac{Cont(\tilde{C}_1,\tilde{C}_2)}{Cont(C_1,C_2)},\frac{Cont(C_1,C_2)}{Cont(\tilde{C}_1,\tilde{C}_2)}\right\}<1.$$
So, it comes from Proposition \ref{contact2} that $C$ is not $\alpha$-H\"older homeomorphic to $\tilde{C}$ with $\alpha_0<\alpha<1$.
\end{proof}

As a consequence of Theorems \ref{main-theorem} and \ref{main-result}, we get the following
\begin{theorem}\label{first-cor}
Let $C$ and $\tilde{C}$ be germs of complex analytic plane curves at $0\in\C^2$. Let $C_1,...,C_r$ and $\tilde{C}_1,...,\tilde{C}_s$ be the branches of $C$ and $\tilde{C}$, respectively. If, for each $\alpha\in (0,1)$, there exists a bi-$\alpha$-H\"older homeomorphism between $C$ and $\tilde{C}$, then there is a bijection $\sigma:\{1,...,r\}\to \{1,...,s\}$ such that
\begin{itemize}
\item [i)] the branches $C_i$ and $\tilde{C}_{\sigma(i)}$ have the sames characteristic exponents, for $i=1,...,r$;
\item [ii)] the pair of branches $(C_i,C_j)$ and $(\tilde{C}_{\sigma(i)},\tilde{C}_{\sigma(j)})$ have the same intersection multiplicity at $0$, for $ i\neq j\in \{1,\dots,r\}$.
\end{itemize}
\end{theorem}
\begin{proof}
Let us remark that, if $h:C\to\tilde{C}$ is a homeomorphism, then by Lemma A.8 in \cite{Gau-Lipman:1983}, as already used in the proof of the Theorem \ref{main-theorem}, for each $u\in \{1,...,r\}$ there is exactly one $j\in \{1,...,s\}$ such that $h(C_u)= \tilde{C}_v$ and, in particular, $r=s$.
Let $E=\{\frac{1}{2}\}\cup\{k=k_{ij}(C_u,\tilde{C_v}); u,v\in\{1,...,r\}\mbox{ and } k<1\}\cup\{k=\min\left\{\frac{Cont(\tilde{C}_u,\tilde{C}_v)}{Cont(C_i,C_j)},\frac{Cont(C_i,C_j)}{Cont(\tilde{C}_u,\tilde{C}_v)}\right\}; i,j,u,v\in\{1,...,r\}, i\not=j, u\not=v \mbox{ and } k<1\}$. We have that $E$ is a finite and non-empty set with $k_0=\max E<1$. Thus, let $\alpha\in (0,1)$ such that $\alpha^4>k_0$ and let $h:C\to \tilde{C}$ be a bi-$\alpha$-H\"older homeomorphism. By Theorem \ref{main-theorem}, for each $i\in \{1,...,r\}$, $C_i$ and $\tilde{C}_{\sigma(i)}=h(C_i)$ have the same characteristic exponents. Moreover, for each $i,j\in \{1,...,r\}$, by Theorem \ref{main-result}, $Cont(C_i,C_j)=Cont(\tilde{C}_{\sigma(i)},\tilde{C}_{\sigma(j)})$. Since $Cont(C_i,C_j)=Cont(\tilde{C}_{\sigma(i)},\tilde{C}_{\sigma(j)})$, it comes from Lemma 3.1 in \cite{F} that the pairs $(C_i,C_j)$ and $(\tilde{C}_{\sigma(i)},\tilde{C}_{\sigma(j)})$ have the same coincidence at $0$ and, therefore, by Proposition 2.4 in \cite{M}, we get that the pairs $(C_i,C_j)$ and $(\tilde{C}_{\sigma(i)},\tilde{C}_{\sigma(j)})$ have the same intersection multiplicity at $0$.
\end{proof}

We are going to show that Theorem \ref{first-cor} generalizes some known results which we list below. For instance, since Lipschitz maps are $\alpha$-H\"older for all $0<\alpha\leq 1$, we obtain, as a first application of Theorem \ref{first-cor}, the main result in \cite{F}.
\begin{corollary}\label{cor-2}
Let $X$ and $Y$ be germs of complex analytic plane curves at $0\in\C^2$. If there exists a bi-Lipschitz subanalytic map between $X$ and $Y$, then $X$ and $Y$ are topologically equivalent.
\end{corollary}
Actually, we do not use the subanalytic hypotheses in Theorem \ref{first-cor}, hence we obtain the following result proved in \cite{N-P}.
\begin{corollary}\label{cor-3}
Let $X$ and $Y$ be germs of complex analytic plane curves at $0\in\C^2$. If there exists a bi-Lipschitz homeomorphism between $X$ and $Y$, then $X$ and $Y$ are topologically equivalent.
\end{corollary}

Since germs of complex analytic curves in $\C^n$ (spatial curves) are bi-Lipschitz homeomorphic to their generic projections, we also get, as a immediate consequence of Theorem \ref{first-cor} the following.
\begin{corollary}\label{cor-3}
Let $X$ and $Y$ be germs of complex analytic curves in $\C^n$ and $\C^m$ respectively. If there exists a bi-$\alpha$-H\"older homeomorphism between $X$ and $Y$, for all $0<\alpha\leq 1$, then $X$ and $Y$ are bi-Lipschitz homeomorphic.
\end{corollary}
\begin{proof} Let $\tilde{X}\subset\C^2$ and $\tilde{Y}\subset\C^2$ be generic projections of $X$ and $Y$ respectively. Since $X$ and $\tilde{X}$ (respectively $Y$ and $\tilde{Y}$) are bi-Lipschitz homeomorphic, it follows that $\tilde{X}$ and $\tilde{Y}$ are $\alpha$-H\"older homeomorphic for all $\alpha\in (0,1)$. By Theorem \ref{first-cor}, there exist a bijection between the branches of $\tilde{X}$ and $\tilde{Y}$ that preserves characteristic exponents of branches and, also, preserves intersection multiplicity of pairs of branches. Hence, using Pham-Teissier Theorem (quoted in the introduction), $\tilde{X}$ and $\tilde{Y}$ come bi-Lipschitz homeomorphic. It finishes the proof.
\end{proof}

We also obtain, in the case of complex analytic plane curves, a generalization of the main result in \cite{BirbrairFLS:2016} and the Theorem 4.2 in \cite{Sampaio:2016}.
\begin{corollary}\label{application2}
Let $X\in \C^n$ be a germ of complex analytic curve at the origin. Suppose that, for each $\alpha\in (0,1)$, there is a bi-$\alpha$-H\"older homeomorphism $h:(X,0)\to (\C,0)$. Then, $(X,0)$ is smooth.
\end{corollary}


We would like to finish this section by stressing the existence of germ of sets that are $\alpha$-H\"older homeomorphic, for all $0<\alpha<1$, but are not bi-Lipschitz homeomorphic.
\begin{definition}
{\rm We say that $h:C\subset \R^n\to \R^m$ is a}  log-Lipschitz map {\rm, if there exists $K>0$ such that} $\|h(x)-h(y)\|\leq K \|x-y\|\cdot|log\|x-y\||$, whenever $x,y\in C$ and $\|x-y\|<1$.
\end{definition}
\begin{remark}
{\rm If $h$ is a log-Lipschitz map, then $h$ is $\alpha$-H\"older, for all $\alpha\in (0,1)$.}
\end{remark}
\begin{definition}
{\rm  Let $(X,x_0)$ and $(Y,y_0)$ be germs of Euclidean subsets. We say that $(X,x_0)$ is \emph{bi-log-Lipschitz homeomorphic to} $(Y,y_0)$ if there exists a germ of homeomorphism $f\colon (X,x_0)\rightarrow (Y,y_0)$ such that $f$ and its inverse $f^{-1}$ are log-Lipschitz mappings. In this case, $f$ is called a \emph{bi-log-Lipschitz homeomorphism} from $(X,x_0)$ onto $(Y,y_0)$.}
\end{definition}
\begin{corollary}\label{application1}
Let $C$ and $\tilde{C}$ be germs of complex analytic plane curves at $0\in\C^2$. If $C$ and $\tilde{C}$ are bi-log-Lipschitz homeomorphic, then they are  bi-Lipschitz homeomorphic.
\end{corollary}

According to the example below, one see that the last corollary is very dependent on the rigidity of analytic complex structure of the sets.

\begin{example}
Let $\tilde{C}=\{(x,y)\in \R^2; y=|xlog|x||\}\cup \{(0,0)\}$. The homeomorphism $h:(\R,0)\to (\tilde{C},0)$ given by $h(x)=(x,|xlog|x||)$ ($h(0)=(0,0)$) is a bi-log-Lipschitz homeomorphism. However, $(\tilde{C},0)$ is not bi-Lipschitz homeomorphic to $(\R,0)$.
\end{example}

\end{document}